\documentclass[10pt,openright]{article}
\usepackage{amssymb, amsmath}
\usepackage{color}
\usepackage{latexsym}
\usepackage{amscd}
\usepackage{amsmath,amsfonts,amssymb,amsbsy,amstext,amsthm,amscd}
\usepackage{graphicx}

\newtheorem{theo}{Theorem}

\newtheorem{propo}{Proposition}
\newtheorem{lema}{Lemma}
\newtheorem{conj}{Conjecture}

\begin{document}

\title{On Malcev algebras nilpotent by  Lie center and  corresponding analytic Moufang loops.}

\date {May, 2020}

\author{Alexander Grishkov, Marina Rasskazova,\\ Liudmila Sabinina, Mohamed Salim}
\maketitle

\begin{abstract}

In this note we describe the structure of finite dimensional Malcev algebras over a field of real numbers $\mathbb{R}$, which are nilpotent module its Lie center. It is proved that the corresponding  analytic global  Moufang loops are nilpotent module their nucleus. \\


\end{abstract}
\section{Introduction}
The theory of analytic loops started with the work of A.I. Malcev in [Ma1]. In this article, the correspondence between the analytic local diassociative loops and the binary Lie algebras was established.
A {\it loop}   is a set $M$, endowed with a  binary operation $M\times M\rightarrow M$, with the neutral element $e \in M$ and the
condition that the equations $ax=b, ya=b$ for all $a,b \in M$ have a unique solution. A  loop is called {\it  diassociative},
if every two elements of this loop generate a subgroup.
Malcev showed that
 for an analytic  loop, with  the Moufang identity:
$(xy)(zx) = x(yz)x$  the corresponding tangent algebra satisfies the following identities:
$$x^2= [J(x,y,z), x] - J(x, y, [x,z])=0,$$
where $J(x,y,z) = [[x, y], z]+[[x, z], x]+[[z, x], y]$ (See [1]). Algebras with these defining identities  are currently called \textit{Malcev algebras}.
The more difficult question, if every finite dimensional  Malcev $\mathbb{R}$-algebra is the tangent algebra of some local analytic Moufang loop  was solved positively by E.Kuzmin in 1969 in [Kuzm2].

Let us consider a pair $(\mathfrak a,\mathfrak L)$, where $\mathfrak a$ is some subvariety of binary Lie  algebras (in particular Malcev algebras)
 and $\mathfrak L$ is a  subvariety of  diassociative loops (in particular, Moufang loops).

A pair $(\mathfrak a,\mathfrak L)$,  will be called \textit{locally dual},  if it  satisfies the following conditions:\\
\begin{itemize}
\item
 every tangent algebra  of a local analytic loop from the variety $\mathfrak L$
belongs to the variety $\mathfrak a$,
\item
 every  finite-dimensional $\mathbb{R}$-algebra from the variety $\mathfrak a$ is a tangent algebra of some local
 analytic loop from the variety $\mathfrak  L$
\end{itemize}

The description of all locally dual  pairs is a meaningful and difficult problem.

If a locally dual  pair $(\mathfrak a,\mathfrak L)$  satisfies a  stronger condition, namely, if  for every  global analytic loop $S$ such
that its local loop belongs to $\mathfrak L$, is in $\mathfrak L$ too, we will call such a pair $(\mathfrak a,\mathfrak L)$ {\it globally dual}.

Kerdman \cite{Ker} showed that a pair $ (\mathfrak m, \mathfrak M)$ is globally dual if $\mathfrak m$ is the variety of all Malcev  algebras, and
$\mathfrak M$ is a variety of all Moufang loops. In this note we  study  the duality of a pair formed by two varieties:
the first  one, $NL_k$, is the subvariety
of the variety $\mathfrak m$, which is defined by 
the identity   $J([x_1, x_2,...,x_k], y, z)=0,$
where in the product $[x_1, x_2,...,x_k]$ a distribution of parentheses is arbitrary
and the second  one, $NG_k$, is  the subvariety
of the variety $\mathfrak M$, defined by  the identity
 $([x_1, x_2,...,x_k], y,z)=1.$  Here $[x_1, x_2,...,x_k]$ is a commutator  of length $k$ with an arbitrary distribution of parentheses
and  $(a,b,c) = ((ab)c)(a(bc))^{-1}$ is an associator.\\
 For a Malcev algebra $M$ we define $M^1=M$, $M^n=\underset { i,j>0} {\overset  {i+j=n} {\sum}} [M^{i},M^{j}]$.
Let $F$ be a free Malcev algebra, then $NL_k$ is a variety of Malcev algebras defined by all identities of the type $J(w ,x, y)=0$, where $w\in F^{k}, \, k\geq 2.$

\section{Structure of finite dimensional Malcev \\algebras nilpotent by Lie center.}

Let $M$ be a finite dimensional Malcev algebra  over a field $\mathbb C, \,\,\, $ and let $G$ be a solvable radical of
$M. $ Then there exists  a semisimple subalgebra (Levi factor)  $S,$ such that  $ M = S \oplus G $
( [Gri1],[Kuzm1],[Car]). \\

 We will use the results and the terminology from [Gri2]. Let $ g\in M$, the element $g$ is said to be
 {\it splitting element} if $g=t+n$, where $n$ is a nilpotent element and $t$ is a {\it semisimple element}, i.e., the right multiplication operator
$R_t$ is diagonalizable and the operators $R_t$ and $R_n$ commute. A Malcev algebra  $M$ is said to be {\it splitting} if all elements  of $M$ are splitting.
If $M$ is  a finite-dimensional splitting Malcev algebra over a field of characteristic 0, then  $M=S\oplus T \oplus N$,  where $S$ is a  semisimple  Levi factor,
$T$ is  an  abelian  subalgebra  of $M$ such  that  each  element  of  $ T$ is  semisimple ({\it toroidal} subalgebra),
and $N$ is  the  nilpotent  radical  of $M.$ Additionally
$[S,T]=0$,
$$N=\sum_{\alpha\in\Delta}\bigoplus N_\alpha ,$$
where $\Delta \subset T^{*}=\mathrm{Hom}_{k}(T,k)$ and
\begin{align}
 N_{\alpha}=\left\{x\in N\mid [x,t]=\alpha(t)x, \quad \forall t\in T\right\}.
\end{align}
Moreover,
 $[N_{\alpha}, N_{\beta} ] \subseteq  N_{\alpha +\beta },$ if $\alpha \neq  \beta, $ and $ \,[N_{\alpha}, N_{\alpha}]\subseteq N_{2\alpha} + N_{-\alpha}.$\\
Since $[T,S]=0$, one has that  $N_{\alpha}$ is an $S$-module and hence $N_{0}=N_{01}\oplus N_{00}$,  $[S,N_{0}]= [S,N_{01}]=N_{01}$
and $[S,N_{00}]=0$. Set
$$M_{11}=\left(S\oplus\sum_{\alpha\in\Delta\backslash 0}\bigoplus N_{\alpha}\right)\oplus N_{01}$$
and let  us denote  by $M_1$ the subalgebra generated by $T\oplus M_{11}$.  Notice that in general, $M_1\neq T\oplus M_{11}$, and
$[N_{01}, N_{00}] \subseteq N_{01},$ $[M_{11}, N_{0}]\subseteq M_{11}$. Hence $M_{1}$ is an ideal.
Every finite-dimensional Malcev algebra $M$ over a field of characteristic $0$ is contained in some splitting Malcev algebra
$\hat M.$ If such  $\hat M$ does not contain intermediate splitting subalgebra, which contains $M$,  then $\hat M$ is called a {\it splitting of }$M$.
Each automorphism of the algebra $M$ extends uniquely to an automorphism of a splitting of $M.$ If $\hat M$  is a splitting of $M$, then ${\hat M}^2 = M^2 $,
and any ideal of the algebra
$M$ is an ideal of the algebra  $\hat M$ and vice verse any ideal of  $\hat M$, which is in $M$ is also the ideal of $M$.
This result is analogous to one for Lie algebras due to A. I. Malcev [Ma2].\\
In what follows in this article the splitting algebra of an algebra $M$ we will denote by $\hat M.$\\[2ex]

 Recall that
$$Lie(M)=\{x\in M |  J(x,M,M)=0\}$$ is the {\it Lie center} of $M$. \\

\begin{lema} In this notation $Lie(M) \subseteq Lie(\hat M).$
\end{lema}
\begin{proof}
By the construction $\hat M = \underset {i=1}{\overset {n}{\bigcup }}M(i),$ where $M(1) = M, \,M(n) = \hat M,\,$ dim$M(i) =$dim$M(i-1) +1,\,\,
M(i) = M(i-1) + {\mathbb R} t_i, \, i \geq 2,$  where $t_ i$ is a semisimple element, i.e. $R_{t_i } : M(i-1) \rightarrow M(i-1)$ is a diagonalizable operator. Moreover
there exists $ x_i \in M(i-1)$, such that $n_i = x_i  -  t_i$ is a nil element, i.e $R_{n_i}$ is a nilpotent operator and
$R_{t_i} \circ R_{n_i}  = R_{n_i} \circ R_{t_i},\,$
$ [n_i,\,t_i ]= 0.$

Now consider $l\in Lie(M(i-1)), \, x\in M(i-1).$ It is sufficient to show that $J(l, x, t_i)= 0.$ Suppose that
\begin{align}
M(i-1) = \underset {\alpha \in \Delta} {\Sigma}{\oplus {M(i-1)_{\alpha}}}
\end{align}
is a Cartan decomposition with respect to the operator $R_{t_i}$  or  $R_{x_i}.$ \\

Without loss of generality one can assert that
$$l\in {Lie(M(i-1)\cap M(i-1)_{\alpha}}, \, x\in M(i-1)_{\beta}.$$
If $\alpha \neq \beta,$  then $J(l, x, t_i) = 0$ due to  the fact that any Malcev algebra is a binary-Lie algebra.\\
Consider the case $\alpha = \beta \neq 0.$\\
The equality $J(l, x, t_i) = 0$ is equivalent to $[x, l] \in M(i-1)_{2\alpha}.$\\

Indeed, suppose that $[x, l] \notin M(i-1)_{2\alpha}.$
Then $[x, l] = [x, l ]_{2\alpha} + [x, l]_{-\alpha},$
where $[x, l ]_{2\alpha} \in M(i-1)_{2\alpha}$  and $0\neq [x, l]_{-\alpha} \in M(i-1)_{-\alpha}.$
Recall that $x_i -n_i =t_i.$\\
$$0= J(l,x,x_i)
= [[x, l]_{2\alpha}, (t_i +n_i)] + [[x, l]_{-\alpha}, (t_i +n_i)]+ [[x, (t_i +n_i)], l] -  [[l, (t_i +n_i)], x] $$
We get
$${2\alpha}[l,x] _{2\alpha}- {\alpha}[l,x]_{-\alpha} +[[x, l]_{2\alpha}  + [x, l]_{-\alpha}, n_i] + [[x, t_i], l]+[[x, n_i], l] - [[l, t_i ], x]  -[[l, n_i], x] $$
$$= {2\alpha}[l,x] _{2\alpha}- {\alpha}[l,x]_{-\alpha} +   [[(x, l ]_{2\alpha} + [x, l]_{-\alpha}), n_i]  +[[x, n_i], l] +2 {\alpha}([x, l ]_{2\alpha} + [x, l]_{-\alpha})$$
$$ - [[l, n_i], x]
=- {2\alpha}[x, l]_{2\alpha} + {\alpha}[x, l]_{-\alpha} + 2 {\alpha}([x, l ]_{2\alpha} + [x, l]_{-\alpha}) +[([x, l ]_{2\alpha} + [x, l]_{-\alpha}), n_i] $$
$$+[[x, n_i], l]_{2\alpha} + [[x, n_i], l]_{-\alpha} - [[l, n_i], x]_{2\alpha} - [[l, n_i], x]_{-\alpha}=0  $$
Finally we get

\begin{align}
3{\alpha}[x, l]_{-\alpha} +[[x, l]_{-\alpha}, n_i] +[[x, n_i], l]_{-\alpha}  - [[l, n_i], x]_{-\alpha}=0
\end{align}

For every element $0\neq a\in {M(i-1)_{\alpha}}$ one can define the number\\  $\left|{a}\right| =min\{m \mid R^m_{n_i}a =0\}.$
Then the equality $[x, l]_{-\alpha}= 0$ one can show by induction on  the sum  $\left|{x}\right| + \left|{l}\right|.$
If $\left|{x}\right| + \left|{l}\right|= 2$ one has
$ [[x, n_i], l]_{-\alpha} =[[l, n_i], x]_{-\alpha}=0.$
Now suppose that  $2<  \left|{x}\right| + \left|{l}\right| = m.$ \\
Then one has $\left|{[x, n_i]}\right| + \left|{l}\right| < m$ and $\left|{[l, n_i]}\right| + \left|{x}\right| < m$\\
and by induction hypothesis
$[[x, n_i], l]_{-\alpha} = [[l, n_i], x]_{-\alpha}=0.$\\
Therefore by (3) we get
\begin{align}
3{\alpha}[x, l]_{-\alpha} +[[x, l]_{-\alpha}, n_i] =0
\end{align}
Now if $\mid [x, l]_{-\alpha} \mid =s$ we can apply  $R^{s-1}_{n_i}$ to the equality (4) and
get $[x, l]_{-\alpha} =0.$\\
 In the case $\alpha = \beta =0$ we have $J(l, x, t_i) = 0,$ since  $[N_0, N_0]\subseteq N_0$ and  $[N_0, t_i] =0$

Thus one has
$$Lie(M) \subseteq ...\subseteq Lie(M (i))\subseteq ... \subseteq Lie(\hat M)$$

\end{proof}

In the theory of Lie algebras there exists  the following construction of {\it decomposable extension}.
Let $L$ be a Lie algebra and let $N$ be a subalgebra of the Lie algebra $\mathrm{Der}L$ (the algebra of all derivations of $L$). Then the direct sum $N\oplus L$ has a structure of Lie algebra with the multiplication:
\begin{align}
(a,l)\cdot (b, r) = ([a,b],\ l^b-r^{a} + [l,r]). \label{IV}
\end{align}

Notice that we are not assuming that $N$ or $L$ is abelian. \\

This construction has a generalization for Malcev algebras.\\

Suppose that $M$ is a Malcev algebra such that $M= \tilde N +L,$ where $L\subseteq Lie(M),$ and $M$ has an ideal $I \in \tilde N$ such that $J(\tilde N)\subseteq I,\,  [I,L]=0,$
then $\tilde N/I  \cong  N$ is a Lie algebra. It means that
 $N$ acts on $L$ by derivations. In this case the formula (\ref{IV}) defines a  Malcev algebra structure on $\tilde M=\tilde N \oplus L$, where $I$ acts trivially on  the
Lie algebra $L$ by definition. This construction is called the {\it decomposable extension of Malcev algebras}. Notice that in this construction
$L$ is an ideal contained in the Lie center of $M$.\\
In what follows the decomposable extension of Malcev algebra  will be denoted by $\tilde M.$\\

The aim of this section is to show the following
\begin {theo}
Let $M$ be a finite dimensional Malcev algebra  from the variety $NL_k$ over a  field of complex numbers $\mathbb C$. Then
\begin {enumerate}
\item $M$ may be embedded into the splitting  Malcev algebra $\hat M = S \oplus T  \oplus N \in NL_k, $ where $S$ is a semi-simple Lie subalgebra,
$T$ is a toroidal subalgebra, $N$ is a nilpotent ideal.
\item $N = N_{00} \oplus [S,N],$ where $ N_{00}$ is a Malcev subalgebra of $N$ and the ideal $M_1$ generated by
$S \oplus T  \oplus  \left(\sum_{\alpha\in\Delta\backslash 0}\bigoplus N_{\alpha}\right)\oplus  [S, N]$ is contained in $Lie (\hat M)$.
\item
There exists a Malcev algebra $\tilde M=N_{00}\oplus M_{1}$  of the variety $NL_k$,  which is a decomposable extension of a nilpotent
Malcev algebra $N_{00}$ and a Lie subalgebra $M_1,$ such that there exists an epimorphism $\pi:\tilde M \longrightarrow \hat M$.
\end{enumerate}
\end {theo}

In order to prove this Theorem we need to collect some intermediate results which we will present in the following lemmata.\\[2ex]

Put  $\underset {n=1}{\overset {\infty} {\bigcap}} M^{i}$ by  $M^{\omega}.$
Following introduced notation one has:

\begin{lema}
Let $M$ be a splitting Malcev algebra from the variety $ NL_k$, $k\geq 2$. Then
ideal $M_1$, constructed above, contains $M^{\omega}$ and is contained in the Lie center $Lie(M).$
\end{lema}
\begin{proof}
 By definition of the variety $NL_k$ we get that $M^{\omega}\in Lie(M)$ the Lie center of $M$.  By construction $M_{11}\subseteq M^{\omega}\subseteq Lie(M)$,
hence $M' \subseteq Lie(M)$ where $M'$ is a subalgebra of $M$ generated by $M_{11}$. It is clear that
$M_1=T\oplus M',\, [T, M']\subseteq M'=S\oplus V$ with $V\subseteq N$.
Hence for proving the lemma it is enough to prove that $J(x, y, z)=0$, where $x,\, y$ and $z$ are elements of $T\cup\left(\cup_{\alpha\in\Delta}N_{\alpha}\right)$.
If $x, y\in T$, then $z\in N_{\alpha}$ and
$$J(x, y ,z) = [[z,x],y]+[[(y,z)],x] = \alpha(x)\alpha(y)z-\alpha(y)\alpha(x)z=0.$$

If $x\in T$, $y\in N_{\alpha}, z\in N_{\beta}$ and $\alpha\neq 0$ or $\beta\neq 0$, then $y\in M_{11}\subseteq Lie(M)$ or $z\in Lie(M)$, hence $J(x, y, z)=0$.
At last, in the case $\alpha=\beta=0$ we get $J(x, y, z)=0$ since $[N_0,N_0]\subseteq N_0$.
 \end{proof}

\begin{lema}\label{l2}
Let $M$ be a Malcev algebra from the variety $NL_k$.

Then $[J(M), M^{\omega}]=0.$
\end{lema}
\begin{proof}
By the result of Filippov (see [Fi],page 236) one has  $[J(M), Lie(M)] = 0$. On the other hand,  $M^{\omega} \subseteq   Lie(M)$ since  $M\in NL_k.$
\end{proof}

Since $M_1\subseteq M^{\omega}$  by Lemma \ref{l2} we have
$[J (M), M_1] = 0.$ Then the subalgebra  $N_{00}$ acts on the ideal $M_1$  by derivations, hence it is possible to construct,
as above, a Malcev algebra $\tilde M= N_{00} \oplus M_1$ with a product given by (\ref{IV}).\\
It is easy to see that the morphism $\varphi: \tilde M\longrightarrow M,\quad \varphi(n, m)=n+m$ is an epimorphism of Malcev algebras.\\

\begin{lema} The Malcev algebra $\tilde M= N_{00} \oplus M_1$  with the product given by (\ref{IV}) is a Malcev algebra of the variety $NL_k$.
\end{lema}
\begin{proof}
Since $M_1\subset Lie(\tilde M),$ one has   $J(\tilde M)= J(N_{00})$.  By construction,\\ $[M_1,   J(N_{00})]=0.$
In this case, $\tilde M$ is a Malcev algebra of the variety $NL_k$ if and only if $N_{00} $ is  a Malcev algebra of the variety $NL_k$; which is exactly our case.\\
\end{proof}

{\bf Remark.} In general, if we have a Malcev algebra $ P= P_0 +P_1,$ where $P_0$ is a nilpotent subalgebra, $P_1 \subseteq Lie(P)$ is
an ideal contained in the Lie center  of $P$ and $P/P_1$ is a Malcev algebra of the variety $NL_k$, then $P$ is not necessarily a Malcev algebra of the the variety $NL_k$.
It is possible that $P\in NL_{k+1}\setminus NL_k.$\\

{\bf Example.} Set $P_1 =  {\mathbb R}\{ t, a, b, c\mid [a, t]=a, [b, t]=-b, [a, b]=c, [c, t]=[a, c]=[b, c]=0\}$ and let it be a splitting Lie algebra. Choose  any nilpotent Malcev algebra  $P_0$ which is not a Malcev algebra from  the variety $NL_k$, but  $P_0/Z$ is.  Here $Z={\mathbb R}z$ is some central ideal  of $P_0$.  It is easy to construct an algebra with those properties.  Let us consider  $\tilde P = P_0 \oplus P_1$ and $P=\tilde P /I$, where $I = {\mathbb R}(c - z)$ is a central ideal. It is clear that $P$ is not a Malcev algebra from the variety $NL_k$.
But $ P =  \pi (P_0) + \pi (P_1), $ where  $\pi : \tilde P \longrightarrow P$ is a canonical homomorphism. Notice that $\pi (P_0)\cong P_0, \pi (P_1)\cong P_1\subseteq N(P)$ and  $P/\pi (P_1) = P_0/Z$ is a Malcev algebra from the variety $NL_k$.\\

{\it Proof of the Theorem.}\\
Consider $\hat M =S \oplus T \oplus N$ as splitting algebra of $M = S  \oplus G.$ We will show that $\hat M \in NL_k$ if $M \in NL_k.$
Due to the construction of $\hat M$ ([Gri2]) any ideal $ I\lhd M$ is also the ideal of $\hat M$  and therefore $ M^k = {\hat M}^k, \, k\geq 2$.
Now since $Lie(M) \subseteq Lie(\hat M)$ and
$M^k \subseteq Lie(M)$ one has $M^k \subseteq Lie(M) \subseteq Lie (M^k).$
This means $M^k \in NL_k.$
As it was noticed above $N$ is a semisimple  $S\oplus T$- module, therefore
$ N_{00} = Ann_N(S\oplus T)$ is the nilpotent subalgebra of $\hat M.$  Moreover
$$M_{11} \subseteq {\hat M^{\omega}} \subseteq Lie(\hat M)$$
Since $[T, N_{00}] = 0$ one has $ N_{00} \subseteq N_0. $
In other hand in general case $N_{00} \cap M_1 \neq 0.$
Recall that
$M_1$ is the ideal  generated by $T \oplus M_{11}.$
Finally one gets  $\tilde M=  N_{00} \oplus M_1.$
\qed
\\[4ex]

\section{Malcev algebras and global  Moufang loops.}

A variety ${\bf M}$ of Malcev algebras will be called \textit{(locally) smooth},  if there exists a variety of Moufang loops $\mathcal{L}$ such that the pair
$({\bf M}, \mathcal{L})$ is (locally) globally dual. Analogously, a variety of Moufang loops $\mathcal{L}$ is \textit{(locally) smooth} if there exists a variety of Malcev algebras  ${\bf M}$, such that the pair
$({\bf M}, \mathcal{L})$ is (locally) globally dual.
A dual pair $(M, \mathcal{L})$ will be called \textit{global} if for any local analytic loop $G$ of the variety $\mathcal{L}$,
there exists a global analytic loop $\tilde G$  from the variety $ \mathcal{L}$ which is locally isomorphic to $G$.
It is clear, that not all varieties of Moufang loops are smooth.
For example,  the variety ${\bf B}_n$ of Moufang loops of exponent $n$ is not smooth, since every analytic Moufang loop of  a positive dimension  is not periodic.
Nevertheless we have the following Conjecture:
\begin{conj}\label{mgl2}
Every dual pair $(M, \mathcal{L})$  of Malcev algebras and their corresponding Moufang loops is global. \\
\end{conj}
Notice that if  the pair
$({\bf M}, \mathcal{L})$ is locally dual and the variety ${\bf M}$ contains only Lie algebras, then all finite dimensional Lie algebras from ${\bf M}$ are solvable.
Indeed, if $M\in {\bf M}$ is not solvable finite dimensional then $M=S\oplus G,$ where $S$ is semisimple Lie subalgebra.
Hence $S$ contains some simple $3-$dimensional Lie subalgebra $L.$ But the corresponding Lie group $G(L)$
contains free subgroup. Hence $\mathcal{L}$ is variety of all groups and ${\bf M}$ is the variety of all Lie algebras.

We will prove the Conjecture \ref{mgl2} for the pairs $(NL_{k}, G_{k})$, where $NL_{k}$ is the variety defined in the last section and $G_{k}$ is a variety of Moufang loops defined by all identities of the type $(w, x, y)=1$, where $w\in F^{k},\, k\geq 2$ and $F$ is an infinite free generated Moufang loop such that  $F^1=F$,  and $F^k$
is the normal subloop generated by $\underset {i=1} {\overset {k-1} \Pi}[F^i, F^{k-i}]$.\\

\begin{propo}\label{p1}
The pair $(NL_{k}, G_{k})$ is dual for any $k\geq 2$.
\end{propo}
\begin{proof}
Let $M$ be a Malcev $\mathbb{R}$-algebra of dimension $n$ of the  variety $NL_{k}$.  Then $M\cong\mathbb{R}^n$.  There exists a small ball $M_{\epsilon}=\{x\in M\ |\  |x|\leq\epsilon \}$, which  is a local Moufang loop with the product given by the Campbell-Hausdorff formula
\begin{align}
x\cdot y := \mathrm{CH}(x,y)=x+y+\frac{1}{2}[x,y]+\cdots .\label{x1}
\end{align}

Notice that the element  $0$ of $M$ is the unit of this local analytic loop.
From (\ref{x1}) we have  that for every subalgebra $P$ of $M$ the corresponding local subgroup is given by
$P_{\epsilon}= P\cap M_{\epsilon}.$ The subgroup $P_{\epsilon}$ is normal if  and only if $P$ is an ideal of $M.$

From (\ref{x1}) we get
\begin{align}
 \{x, y\}= x^{-1}\cdot y^{-1}\cdot x\cdot y=[x, y]+\sum^{\infty}_{s=3}a_s(x,y),  \label{x2}
\end{align}
where  $a_s(x,y)\in M^{s}$ if $x,y\in M$.  Hence every commutator $w$ in the local Moufang loop $(M, \cdot)$ of length $k\geq 3$ has the form
$w=\underset {i=k} {\overset  {\infty}\sum} w_i$, with $w_i\in M^{i}$. Since $M\in NL_{k}$, then $M^s\subseteq Lie(M)$  for $s\geq k$.
Hence the corresponding commutator subloop $M^k_{\epsilon}$ of local Moufang loop $M_{\epsilon}$ is contained in $Lie(M).$\\

E.Kuzmin proved \cite{Kuzm}, that in a local Moufang loop $(M,\cdot)$ the associator can be expressed as:

\begin{align}
 (x, y, z)=\frac{1}{6}J(x, y, z)+\underset {i=7} {\overset  {\infty}\sum}a_i(x, y, z), \label{x3}
\end{align}
where $a_i(x, y, z)$ is an element of degree $i$ of the ideal $J(M) \subset M.$

By (\ref{x3}) we get that $Lie(M)\cap M_{\epsilon}\subseteq Nuc(M_{\epsilon}),$ where $Nuc(M_{\epsilon})=\{x\in M_{\epsilon}| (x, a, b)=0,\forall a, b\in M_{\epsilon}\}.$
Hence $M^k_{\epsilon}\subset Lie(M)\cap M_{\epsilon}\subset Nuc(M^k_{\epsilon}).$  It means that $(M_{\epsilon},\cdot)\in G_k.$

Now suppose that $(M_{\epsilon},\cdot)\in G_k.$  Following the previous notation, we have:
 \begin{lema}\label{l1}
 $M^k\cap M_{\epsilon}=(M_{\epsilon},\cdot)^k, $ where $(M_{\epsilon},\cdot)^k$ is a commutator subloop of the local loop $(M_{\epsilon},\cdot)$
of degree $k.$
 \end{lema}
 {\bf Proof.}
 From the construction of the local loop $(M_{\epsilon},\cdot)$,   for every ideal $I$ of the Malcev algebra $M$ there is a corresponding  normal subloop
$I_{\epsilon}=I\cap M_{\epsilon}$ of $(M_{\epsilon},\cdot).$
 It is clear that for nilpotent of class $k$  Malcev algebra  $M/M^k$ the corresponding local Moufang loop is $((M/M^k)_{\epsilon},\cdot)$, which is nilpotent of class $k.$ Hence $(M_{\epsilon},\cdot)^k\subseteq M^k$.

Suppose that  $(M_{\epsilon},\cdot)$ is a  nilpotent local  loop of class $k.$ By induction we prove that the Malcev algebra $M$ is nilpotent of the same class $k$.
It is clear  for $k=1$. If the Malcev algebra $M$ is not nilpotent of degree $k$ then for some $x_1,...,x_k\in M$ we have
 $w=[x_1,...,x_k]\not=0$  for some distribution of parentheses.   By (\ref{x2}) we get in $(M_{\epsilon},\cdot)$: $u_t=\{tx_1,tx_2, ...,tx_k\}=t^kw+\sum_{i>k}t^{i}w_i,$ where
 $w_i$ is an element of $M^i$ and $ t\in \mathbb{R}.$ Since $(M,\cdot)$ is nilpotent of degree $k$, $u=0$ in $(M,\cdot).$ Then $w=0$ in $M$ which yields to a contradiction.\\
\qed

With all considerations above Lemma \ref{l1} is proved.

 Now we can finish the proof of
 Proposition \ref{p1}. Let $w=[x_1,...,x_k]\in M^k$, we have to prove that $w\in Lie(M).$ For some $t\in\mathbb{R} $ we have  by Lemma \ref{l1} that $t^kw=[tx_1,....,tx_k]\in M^k\cap M_{\epsilon}\subset (M_{\epsilon},\cdot)^k\subset Nuc(M_{\epsilon}).$ Here we used that $(M_{\epsilon},\cdot)\in G_k.$
 Hence $(t^kw, x, y)=0$ for all $x,y\in M_{\epsilon}.$ By \ref{x3} we get that
 $J(t^kw, x, y)=0.$ It means that $w\in Lie(M).$
 Proposition \ref{p1} is proved.
 \end{proof}

Now we are ready to prove the main result of this paper.\\

\begin{theo}
The dual pair $ (NL_k, G_k)$ is global.
\end{theo}
\begin{proof}
Let $G_0\in G_k$ be a local analytic loop and let $L(G_0)=M\in NL_k$ be its corresponding Malcev algebra.  Let
$\varphi: M \hookrightarrow \hat{M}$ be an embedding   of $M$ in a splitting Malcev algebra $\hat{M}=S\oplus T\oplus N$, (see notation of Theorem 1).
By definition $\hat{M}$ is minimal with this property. Hence $[\hat{M}, \hat{M}]=[M,M].$  Since $M\in NL_k$ then $\hat{M}\in NL_k$ by Theorem 1.
By \cite{Ker} there exists the corresponding to $\hat{M}$ global analytic simply connected  Moufang loop $\hat{G}.$

By construction of $\hat{G}$ in \cite{Ker} we get that
$\hat{G}=P\cdot Q,$ where $P$ is simply connected semisimple Lie group with corresponding Lie algebra $S$ and $Q=Q_0\cdot Q_1$ is simply connected
solvable Moufang loop with corresponding
Malcev algebra $T\oplus N,$ and $Q_0\simeq \mathbb{R}^t\simeq T$ is abelian vectorial group Lie, corresponding to Lie  subalgebra $T,$ $Q_1$ is simply connected nilpotent normal subloop corresponding to the nilpotent ideal $N.$
By Theorem 1 we have that $\hat{M}=S\oplus T\oplus N=N_{00}+M_1,$ where $N_{00}\subseteq N$ is a nilpotent subalgebra and $M_1$ is an ideal of $\hat{M}$ that is contained in $Lie(\hat{M}).$ Since exponential map from $N$ to $Q_1$
is a bijection then $Exp(N_{00})=Q_2\subset Q_1$ is a nilpotent simply connected subloop of $\hat{G}.$ Since
$\hat{M}=N_{00}+M_1$ then $\hat{G}=Q_2.G_1,$
where $G_1$ is simply connected normal group Lie corresponding to the ideal $M_1\subset Lie(\hat{M}).$

  Since $\hat{M}$ is a splitting  by Theorem 1 there exists a Malcev algebra $\tilde{M}=N_{00}\oplus M_1\in NL_k$   and an epimorphism
$\pi: \tilde{M}\longmapsto \hat{M}.$\\
  Let $\tilde{G}$ be simply connected analytic
 Moufang loop corresponding to Malcev algebra $\tilde{M}$. Then $\tilde{G}=Q_2\times G_1,$
 where $Q_2$ and $G_1$ are subloops of $\hat{G}.$
 Notice that multiplication in $\tilde{G}$
 may be given the following analogue  of  (3).
 \begin{align}
(r_1, g_1)\cdot(r_2, g_2)=(r_1r_2,  g_1^{r_2}g_2) \label{H1}
\end{align}
  Where $g_1^{r_2}=r_2^{-1}g_1r_2$ is natural action of
  Moufang loop $Q_2$ on the Lie group $G_1$ by automorphisms, since $A(Q_2)=J(N_{00})$ acts trivially on $G_1.$
Here $A(Q_2)$ is an associator of $Q_2$ and we used that $[J(\hat{M}),M_1]\subseteq
  [J(\hat{M}),Lie(\hat M)]=0$.

 It is clear that $G_1$ is contained in the nucleus  of $\tilde{G}$ and
$\tilde{G}\in G_k$ if and only if $Q_2\in G_k$.

 \begin{lema}\label{l5}
 Let $N$ be a nilpotent finite dimensional Malcev algebra and $R$ be the corresponding simply connected analytic Moufang loop.

 If the corresponding local analytic Moufang loop $R_{\epsilon}$ satisfies some identity $f(x_1,\dots, x_n)=1,$ then the global analytic loop
 $R$ satisfies the same identity.

 In particular, $N\in NL_k$ if and only if $R\in G_k.$
 \end{lema}
 \begin{proof} It is possible to assume  that $R=N\simeq {\bf R}^m$ with multiplication (\ref{x1}).
 Let $f$ be an identity of the local analytic loop $R_{\epsilon}.$ Then $f(x_1,\dots,x_n)=\sum_jf_j,$ where $f_j=f_j(x_1,\dots,x_n)$
 is a Lie word in $x_1,\dots,x_n.$ Let $v_1,\dots,v_m$ be a basis of $N.$
 Then
 $f=\sum_{j=1}^m g_jv_j,$  where $g_j=g_j(x_1,\dots,x_n)$ is a polynomial function in $x_1,\dots,x_n,$ $j=1,\dots,m.$
 Since $f$ is a local identity then $g_j=0$ if $|x_s|<\varepsilon,$ $s=1,\dots,n$ and $\varepsilon$ is small enough.
But any polynomial function, which is equal to zero in some neighborhood of $\bar{0}$ is equal to zero for all values of the variables.
Hence $f=1$ is an identity of the loop $R$.
 \end{proof}

 Returning to the proof of the Theorem 2 we have that $\tilde{G}\in G_k$ due to Lemma \ref{l5}. Hence $\hat{G}\in G_k,$
 since $\hat{G}$ is a homomorphic image of $\tilde{G}.$\\
With this the proof of the Theorem 2 is done.
 \end{proof}

\section*{Acknowledgements}
The first author was supported by grants   CNPq  307824/2016-0 and
FAPESP  2018/ 23690-6.
The second author was supported by grant FAPESP  2018/11292-6.
The third author thanks for the support to  FAPESP  for grant 2019/24418-0 and the University of Sao Paulo.
The first and fourth authors was supported by UAEU UPAR grant G00002599.

\end{document}